\documentclass[12pt]{amsart}

\usepackage[all,cmtip]{xy}
\usepackage{graphicx}
\usepackage{amsmath}
\usepackage{amsfonts}
\usepackage{amssymb}
\usepackage{amsthm}
\usepackage{amscd}
\usepackage{textcomp}
\usepackage{hyperref}
\usepackage{extarrows}

\DeclareFontEncoding{OT2}{}{} 
  \newcommand{\textcyr}[1]{%
   {\fontencoding{OT2}\fontfamily{wncyr}\fontseries{m}\fontshape{n}%
     \selectfont #1}}
\newcommand{\Sha}{{\mbox{\textcyr{Sh}}}}
\newcommand{\Cok}{{\mbox{\textcyr{B}}}}

\newcommand{\ssstyle}{\scriptscriptstyle}

\newcommand{\Z}{\mathbb{Z}}

\newcommand{\isom}{\cong}

\renewcommand{\a}{\mathfrak{a}}
\DeclareMathOperator{\Am}{Am}

\DeclareMathOperator{\coker}{coker}
\DeclareMathOperator{\Br}{Br}

\DeclareMathOperator{\Nm}{Nm}

\DeclareMathOperator{\ord}{ord}

\DeclareMathOperator{\st}{st}

\DeclareMathOperator{\Cl}{Cl}

\theoremstyle{plain}
\newtheorem{theorem}{Theorem}[section]
\newtheorem{proposition}[theorem]{Proposition}
\newtheorem{corollary}[theorem]{Corollary}

\newtheorem{lemma}[theorem]{Lemma}

\theoremstyle{definition}

\theoremstyle{remark}

\numberwithin{equation}{section}

 1

\newcommand{\thetitle}
{Capitulation, unit groups, and the cohomology of $S$-id\`{e}le classes}

\makeatletter
      \def\@setcopyright{}
      \def\serieslogo@{}
      \makeatother

\begin{document}

\title{\thetitle} 
\author{Saikat Biswas}
\address{School of Mathematical and Statistical Sciences, Arizona State University,
Tempe, AZ}
\email{Saikat.Biswas@asu.edu}

\begin{abstract}
Let $L/K$ be a finite, cyclic extension of number fields with Galois group $G$, and let $S$ be a finite set of primes of $K$ that includes all the infinite primes. In this paper, we study the $G$-cohomology of the $S$-id\`{e}le classes of $L$ and relate it to the $S$-capitulation map $j_{\ssstyle{L/K,S}}:\Cl_{K,S}\to\Cl_{L,S}^G$ as well as to the $G$-cohomology of the $S$-unit group $U_{L,S}$.
\end{abstract}



\maketitle


\section{Introduction}

Consider a finite, Galois extension $L/K$ of number fields having Galois group $G$. Let $S$ be a finite set of primes of $K$ containing the set $S_{\infty}$ of all infinite primes of $K$. We also denote by $S$ the set of primes of $L$ that divide the primes of $K$ contained in $S$. Let $U_{K,S}$ and $\Cl_{K,S}$ denote, respectively, the $S$-unit group and the $S$-ideal class group of $K$, and let $U_{L,S}$ and $\Cl_{L,S}$ denote the corresponding groups of $L$. The extension of ideals from $K$ to $L$ induces the \emph{$S$-capitulation map} 
$$j_{\ssstyle{L/K,S}}:\Cl_{K,S}\to\Cl_{L,S}^{G}$$
The kernel of $j_{\ssstyle{L/K,S}}$, also called the \emph{$S$-capitulation kernel} of $L/K$, corresponds to $S$-ideal classes in $K$ that become principal (i.e. capitulate) in $L$. On the other hand, the ideal classes in $\Cl_{L,S}^{G}$ are also known as the \emph{ambiguous $S$-ideal classes} in $L$. Thus the cokernel of $j_{\ssstyle{L/K,S}}$ classifies the ambiguous $S$- ideal classes of $L$, up to equivalence, that do not arise from the $S$-ideal classes in $K$.When $S=S_{\infty}$, we usually drop it from the notation.

The forerunner of all results on capitulation is Hilbert's Theorem 94 \cite{hilbert} according to which if $L/K$ is a finite, cyclic, and unramified extension of number fields, then the degree $[L:K]$ divides $\ker{j_{\ssstyle{L/K}}}$, i.e. there are at least $[L:K]$ ideal classes in $K$ that capitulate in $L$. This eventually led to the classical Principal Ideal Theorem of class field theory, conjectured by Hilbert and proved by F\"{u}rtwangler, which states that every ideal class in $K$ capitulates in the Hilbert class field of $K$. We refer to \cite[Thm II.5.8.3]{gras} for a generalized version of this theorem. Miyake conjectured in \cite{miya} that if $L/K$ is an Abelian, everywhere unramified (including infinity) extension then at least $[L:K]$ ideal classes in $K$ capitulate in $L$. This assertion, which generalizes both Hilbert's Theorem 94 and the Principal Ideal Theorem, was proved by Suzuki in \cite{suzuki} and was further generalized by Gruenberg and Weiss in \cite{gw}. We refer to \cite{schmit}, \cite{gras}, and \cite{miya} for a survey of results pertaining to the capitulation problem. While most of these results deal predominantly with the kernel of the capitulation map, the corresponding results about the cokernel of the capitulation map seem to be somewhat sparse with the possible exception of \cite{gon} which presents results on the cokernel of the $S$-capitulation map.

In this paper, we first present an id\`{e}le-theoretic interpretation for the kernel and cokernel of the $S$-capitulation map. Let $I_{K,S}$ be the $S$-id\`{e}le group of $K$. The group $I_{K,S}$ contains the group $U_{K,S}$, and the quotient $C_{K,S}:=I_{K,S}/U_{K,S}$ is called the group of $S$-id\`{e}le classes of $K$. Similarly, we can define the group $I_{L,S}$ as well as the quotient group $C_{L,S}=I_{L,S}/U_{L,S}$. There is a $G$-module structure on both $I_{L,S}$ and $U_{L,S}$, which in turn induces a $G$-module structure on $C_{L,S}$. However, there is a failure of Galois descent on $C_{L,S}$ since $C_{K,S}$ is not always the fixed module $C_{L,S}^{G}$. We prove that

\begin{theorem}\label{mt1}
Suppose $L/K$ is a finite, cyclic extension of number fields with Galois group $G$. Then there are isomorphisms
\begin{align*}
\ker{j_{\ssstyle{L/K,S}}} &\isom C_{L,S}^{G}/C_{K,S}\\
\coker{j_{\ssstyle{L/K,S}}} &\isom H^1(G,C_{L,S})
\end{align*}
In particular, the $S$-capitulation kernel of $L/K$ measures the failure of Galois descent for $C_{L,S}$, while $H^1(G,C_{L,S})$ measures the failure of Galois descent for $\Cl_{L,S}$ when Galois descent holds for $C_{L,S}$.
\end{theorem}

It follows immediately that

\begin{corollary}
The index $[C_{L,S}^{G}:C_{K,S}]$ divides $h_{K,S}$, where 
$$h_{K,S}:=[\Cl_{K,S}]$$ 
is the $S$-class number of $K$.
\end{corollary}

Let $v$ be any prime of $K$ and $w$ be a prime of $L$ that divides $v$. Let $K_v$ and $L_w$ be the completions of $K$ and $L$ at the corresponding primes, and let $U_v$ and $U_w$ be the unit groups of $K_v$ and $L_w$, respectively. Let $G_w$ be the decomposition group of $w$ over $K$, also considered as the Galois group of $L_w/K_v$. There is a $G_w$-module structure on both $L_w^{\times}$ and $U_w$, so that the groups $H^i(G_w,L_w^{\times})$ and $H^i(G_w,U_w)$ are defined.
The group $U_{L,S}$ has the structure of a $G$-module, and we consider the localization maps
$$\lambda_{S}^i\;:\;H^i(G,U_{L,S}) \longrightarrow\left( \bigoplus_{v\in S}H^i(G_w,L_w^{\times})\times\bigoplus_{v\notin{S}}H^i(G_w,U_w)\right)$$

We define the groups $\Sha_{S}^i(G,U_{L,S})$ and $\Cok_{S}^i(G,U_{L,S})$ by 
\begin{align*}
\Sha_{S}^i(G,U_{L,S}) &=\ker{\lambda_{S}^i}\\
\Cok_{S}^i(G,U_{L,S}) &=\coker{\lambda_{S}^i}
\end{align*}

In this paper, we present results pertaining to $\Sha_{S}^i(G,U_{L,S})$ and $\Cok_{S}^i(G,U_{L,S})$ for $i=1,2$. To begin with, we show that

\begin{theorem}\label{mt2}
There are isomorphisms
\begin{align*}
\Sha_{S}^1(G,U_{L,S}) &\isom C_{L,S}^{G}/C_{K,S}\\
\Sha_{S}^2(G,U_{L,S}) &\isom\left(U_{K,S}\cap\Nm(L^{\times})\right)/\Nm(U_{L,S})
\end{align*}
where $U_{K,S}\cap\Nm(L^{\times})$ is the subgroup of $S$-units in $K$ that are norms of elements of $L$.
\end{theorem}

The first isomorphism in Theorem \ref{mt2} implies that $C_{L,S}^{G}/C_{K,S}$ may be identified with the subgroup of $H^1(G,U_{L,S})$ that are locally trivial at all primes outside $S$. In particular, Galois descent holds for $C_{L,S}$ if and only the local-global principle holds for $U_{L,S}$.

As noted earlier, the ideal classes in $\Cl_{L,S}^{G}$ are also known as the \emph{ambiguous} $S$-ideal classes. If $\sigma$ is a generator of the cyclic group $G$, then an ambiguous ideal $\a\in\Cl_{L,S}^{G}$ is called \emph{strongly ambiguous} if ${\a}^{\sigma-1}$ is the group of $S$-ideals of $L$. Denoting by $\Am_{\st}(L/K,S)$ the group of strongly ambiguous $S$-ideal classes of $L$, it can be shown that the $S$-capitulation map induces the map
$$j'_{\ssstyle{L/K,S}}:\Cl_{K,S}\to\Am_{\st}(L/K,S)$$
It is known that $\ker{j'_{\ssstyle{L/K,S}}}\isom\ker{j_{\ssstyle{L/K,S}}}$. We prove, on the other hand, that

\begin{theorem}\label{mt3}
There is an isomorphism
$$\coker{j'_{\ssstyle{L/K,S}}}\isom\Cok_{S}^1(G,U_{L,S})$$
Moreover, there is an exact sequence
$$1 \to \Cok_{S}^1(G,U_{L,S}) \to H^1(G,C_{L,S}) \to \Sha_{S}^2(G,U_{L,S}) \to 1$$
\end{theorem}

Thus, the group $\Cok_{S}^1(G,U_{L,S})$ classifies the strongly ambiguous $S$-ideal classes in $L$, up to equivalence, that do not arise from the $S$-ideal classes in $K$. Furthermore, it follows from Theorems \ref{mt1} and \ref{mt3} that the group $\Sha_{S}^2(G,U_{L,S})$ measures the difference between the cokernels of the maps $j_{\ssstyle{L/K,S}}$ and $j'_{\ssstyle{L/K,S}}$.

As for the group $\Cok_{S}^2(G,U_{L,S})$, we prove that

\begin{theorem}
There is an isomorphism
$$\Cok_{S}^2(G,U_{L,S})\isom C_{K,S}/\Nm(C_{L,S})$$
\end{theorem}

We also compute the orders of $\Cok_{S}^i(G,U_{L,S})$ for $i=1,2$.

\begin{theorem}\label{mt4}
We have
\begin{align*}
[\Cok_{S}^1(G,U_{L,S})] &=\frac{[C_{L,S}^{G}:C_{K,S}]\;e_{\ssstyle{L/K,S}}}{[L:K][U_{K,S}:\Nm(U_{L,S})]}\\
[\Cok_{S}^2(G,U_{L,S})] &=[C_{K,S}:\Nm(C_{L,S})]\\
&=\frac{e_{\ssstyle{L/K,S}}}{[U_{K,S}:U_{K,S}\cap\Nm(L^{\times})]}
\end{align*}
where
$$e_{\ssstyle{L/K,S}}=\prod_{v\in{S}}n_v\;\prod_{v\notin{S}}e(v)$$
with $n_v=[L_w:K_v]$ and $e(v)$ the ramification index of $v$ in $L$.
\end{theorem}

As we prove in this paper, the first equality in Theorem \ref{mt4} can be used to establish a generalization of Hilbert's Theorem 94, giving a lower bound on the size of the $S$-capitulation kernel. We also prove a dual result giving a lower bound on the size of the $S$-capitulation cokernel. Both results are also proved in \cite{gon}. As for the second equality in Theorem \ref{mt4}, we note that the $S$-units in $U_{K,S}\cap\Nm(L^{\times})$ coincide with the units that are local norms everywhere outside $S$. Hence, the index $[U_{K,S}:U_{K,S}\cap\Nm(L^{\times})]$ can be computed purely locally, and likewise for the group $\Cok_{S}^2(G,U_{L,S})$.

In section 2, we state and prove some classical results on the group of $S$-ideals, the group of $S$-units, and the group of $S$-id\`{e}le classes. In particular, we prove the $S$-version of Chevalley's ambiguous class number formula. In section 3, we state and prove the main results of this paper. We begin by relating the $S$-id\`{e}le classes to the $S$-capitulation map. We then relate the cohomology of the $S$-id\`{e}le classes to the cohomology of the $S$-units. Finally, we establish a result on the $S$-units and $S$-ideals when $S$ is sufficiently large.

\section{Preliminaries}
Throughout, we fix a finite, cyclic extension $L/K$ of number fields with Galois group $G$. We also fix a finite set $S$ of primes of $K$ that contains the set $S_{\infty}$ of infinite primes. In this setting, we state and prove some preliminary results that will be needed for proving the main theorems.

Let $J_K$ be the group of all fractional ideals of $K$ and $J_{K,S}$ be the group of fractional ideals of $K$ generated by the primes outside $S$. In particular, we have
$$J_K\isom\bigoplus_{v}\Z\;\;\;\textrm{and}\;\;\;J_{K,S}\isom\bigoplus_{v\notin{S}}\Z$$
Let $U_K$ and $\Cl_K$ be, respectively, the unit group and the ideal class group of $K$. These groups are defined by the exactness of the sequence
$$1 \to U_K \to K^{\times} \to J_K \to \Cl_K \to 1$$
where the map $K^{\times}\to J_K$ associates to each $\alpha\in{K^{\times}}$ the principal ideal $(\alpha)\in J_K$. On the other hand, the group of $S$-units of $K$, denoted by $U_{K,S}$, and the $S$-ideal class group of $K$, denoted by $\Cl_{K,S}$, are defined by the exactness of the sequence 
$$1 \to U_{K,S} \to K^{\times} \to J_{K,S} \to \Cl_{K,S} \to 1$$
We denote by $h_K:=[\Cl_K]$ the class number of $K$ and by $h_{K,S}:=[\Cl_{K,S}]$ the $S$-class number of $K$.
There is a natural projection map $J_K\to J_{K,S}$ and applying \cite[Cor 2]{snake} to the composition $K^{\times}\to J_K\to J_{K,S}$, we get the exact sequence
$$1 \to U_K \to U_{K,S} \to \bigoplus_{v\in{S}}\Z \to \Cl_K \to \Cl_{K,S} \to 1$$
We know that $U_K$ is finitely generated, by Dirichlet's Unit Theorem \cite[\S18 Theorem]{global}. Furthermore, the group $\Cl_K$ is finite by \cite[\S17 Theorem]{global}. It follows from the above exact sequence that

\begin{lemma}\label{lem1}
The group $U_{K,S}$ is finitely generated and the group $\Cl_{K,S}$ is finite.
\end{lemma}

The $S$-id\`{e}le group of $K$ is defined as
$$I_{K,S}=\prod_{v\in{S}}K_{v}^{\times}\times\prod_{v\notin{S}}U_{v}$$
Here, $K_v$ denotes the completion of $K$ at $v$ and $U_v$ is the group of units in $K_v$. Thus $I_{K,S}$ is the group of id\`{e}les of $K$ whose components are units at all primes outside $S$. Let $I_K$ be the full id\`{e}le group of $K$. It is defined as
$$I_K=\bigcup_{S}I_{K,S}$$
where the union is over the set of all finite set $S$ of primes of $K$. The next proposition determines the quotient $I_K/I_{K,S}$.

\begin{proposition}\label{p1}
There is a canonical exact sequence
$$1 \to I_{K,S} \to I_K \xrightarrow{\pi} J_{K,S} \to 1$$
\end{proposition}

\begin{proof}
Let $v\notin{S}$ be a prime of $K$ and let $\ord_v$ be the normalized valuation of $K_{v}$. For an id\`{e}le $x=(x_v)\in{I_K}$, we have $x_v\in{U_{v}}$ for almost all $v\notin{S}$, so that $\ord_v(x_v)=0$. Hence, for all primes $v\notin{S}$, the map
$$x=(x_v)\to\prod_{v\notin{S}}{v}^{\ord_v(x_v)}$$
defines a canonical homomorphism $\pi$ from $I_K$ onto $J_{K,S}$. Furthermore, $x\in\ker{\pi}$  if and only if $\ord_v(x_v)=0$ for every $v\notin{S}$, i.e., if and only if the components of $x$ are units at all primes outside $S$, thus if and only if $x\in{I_{K,S}}$.
\end{proof}

\begin{theorem}\label{t1}
There is an exact sequence 
$$1 \to U_{K,S} \to I_{K,S} \to C_K \to \Cl_{K,S} \to 1$$
where $C_K$ is the id\`{e}le class group of $K$.
\end{theorem}

\begin{proof}
By \cite[VII \S3]{ant}, $K^{\times}$ is diagonally embedded in $I_K$, and the id\`{e}le class group $C_K$ of $K$ is defined by the exactness of the sequence
$$1 \to K^{\times} \to I_K \to C_K \to 1$$
Consider now the commutative diagram
\[
\xymatrix{
1\ar[r] &K^{\times}\ar[r]\ar[d]^{\beta} &I_K\ar[r]\ar[d]^{\pi} &C_K\ar[r]\ar[d] &1\\
1\ar[r] &J_{K,S}\ar[r]^{\isom} &J_{K,S}\ar[r] &1 &}
\]
Snake Lemma yields the exact sequence
$$1 \to \ker(\beta) \to \ker(\pi) \to C_K \to \coker(\beta) \to \coker(\pi) \to 1$$
Substituting the kernels and cokernels of $\beta$ and $\pi$, we obtain the requisite exact sequence.
\end{proof}

Thus, the group $I_{K,S}$ contains the $S$-unit group $U_{K,S}$. The quotient, denoted by
$$C_{K,S}=I_{K,S}/U_{K,S},$$
is called the group of \emph{$S$-id\`{e}le classes} of $K$. It immediately follows from Theorem \ref{t1} that

\begin{theorem}\label{t2}
The group of $S$-id\`{e}le classes of $K$ is contained in the id\`{e}le class group of $K$, and we have an isomorphism
$$\Cl_{K,S}\isom C_K/{C_{K,S}}$$
In particular, we have 
$$[C_K:C_{K,S}]=h_{K,S}$$
\end{theorem}

The $S$-id\`{e}le group of $L$ is likewise defined as
$$I_{L,S}=\prod_{w\in{S}}L_{w}^{\times}\times\prod_{w\notin{S}}U_{w}=\prod_{v\in{S}}\left(\prod_{w|v}L_{w}^{\times}\right)\times\prod_{v\notin{S}}\left(\prod_{w|v}U_{w}\right)$$
Consequently, we have
$$H^i(G,I_{L,S})\isom\prod_{v\in{S}}H^i(G,\prod_{w|v}L_{w}^{\times})\times\prod_{v\notin{S}}H^i(G,\prod_{w|v}U_{w})$$
The action of $G$ on $\prod_{w|v}L_{w}^{\times}$ permutes the factors , and the subgroup of $G$ that carry a given factor $L_{w}^{\times}$ into itself is the decomposition group $G_{w}$ of $w$. It follows that $\prod_{w|v}L_{w}^{\times}$ is the $G$-module induced by the $G_{w}$-module $L_{w}^{\times}$. Semilocal theory \cite[IX \S1]{ant} then shows that, for any fixed prime $w$ dividing $v$, we have canonical isomorphisms
$$H^i(G,\prod_{w|v}L_{w}^{\times})\isom H^i(G_{w},L_{w}^{\times})$$ and
$$H^i(G,\prod_{w|v}U_{w})\isom H^i(G_{w},U_{w})$$
Note that $H^2(G_{w},L_{w}^{\times})=\Br(L_w/K_v)$ has order $n_{v}=[L_{w}:K_{v}]$ by \cite[\S1, Cor. 2]{lcft}, while $H^1(G_{w},L_{w}^{\times})$ is trivial by Hilbert's Theorem 90. On the other hand, $H^2(G_w,U_w)$ has order $e(v)$, the ramification index of $v$ in $L_w$, by \cite[XI, \S4, Cor to Thm 4]{ant}. Next, the exact sequence
$$1\to U_w \to L_w^{\times} \xrightarrow{\ord_w} \Z \to 1$$
of $G_w$-modules induces the exact sequence
$$1 \to U_v \to K_v^{\times} \xrightarrow{\ord_w} \Z \to H^1(G_w,U_w) \to 1$$
where the $1$ on the right follows from Hilbert's theorem 90. Since the cokernel of the map $K_v^{\times}\xrightarrow{\ord_w}\Z$ has order $e(v)$, it follows that $H^1(G_w,U_w)$ has order $e(v)$ as well. We have thus shown that

\begin{theorem}\label{t3}
There is a canonical isomorphism
$$H^i(G,I_{L,S})\isom\bigoplus_{v\in{S}}H^i(G_{w},L_{w}^{\times})
\times\bigoplus_{v\notin{S}}H^i(G_{w},U_{w})$$
\end{theorem}

Note that $H^0(G_{w},L_{w}^{\times})=K_{v}$ and $H^0(G_{w},U_{w})=U_{v}$, which immediately imply that

\begin{lemma}\label{l1}
$I_{L,S}^{G}\isom I_{K,S}$
\end{lemma}

We denote the order of the groups $H^i$ by $h_i$ (if they are finite) and the Herbrand quotient by $h_{2/1}=h_2/h_1$. It follows from Theorem \ref{t3} that

\begin{theorem}\label{t4}
$$h_{2/1}(I_{L,S})=\prod_{v\in{S}}n_{v}$$
\end{theorem}

By Theorem \ref{t1}, we have an exact sequence of $G$-modules
$$1 \to U_{L,S} \to I_{L,S} \to C_{L,S} \to 1$$
so that, by \cite[\S8, Prop. 10]{aw}, we have
$$h_{2/1}(C_{L,S})=\frac{h_{2/1}(U_{L,S})}{h_{2/1}(I_{L,S})}$$
Substituting from Theorem \ref{t4} and \cite[IX \S5, Cor 2]{ant}, we get

\begin{theorem}\label{t5}
$$h_{2/1}(C_{L,S})=[L:K]$$
\end{theorem}

We now use the id\`{e}le-theoretic results obtained thus far to prove some known ideal-theoretic results. We begin by measuring the failure of Galois descent for the $S$-ideal group $J_{L,S}$.

\begin{lemma}\label{lem3}
$$[J_{L,S}^{G}:J_{K,S}]=\prod_{v\notin{S}}e(v)$$
\end{lemma}

\begin{proof} 
By Proposition \ref{p1}, we have an exact sequence
$$1 \to I_{L,S} \to I_L \to J_{L,S} \to 1$$
of $G$-modules. The induced long exact sequence of cohomology yields
$$1 \to I_{L,S}^{G} \to I_L^G \to J_{L,S}^{G} \to H^1(G,I_{L,S}) \to H^1(G,I_L)$$
Note that
\begin{align*}
I_{L,S}^{G} &\isom I_{K,S}\;\; \textrm{by Lemma \ref{l1}}\\
I_L^G &\isom I_K \;\;\;\;\textrm{by \cite[\S7, Prop 7.3]{gcft}}\\
\textrm{and}\;H^1(G,I_L)&\;\;\textrm{is trivial by \cite[\S7, Cor 7.4]{gcft}}
\end{align*}
Hence, the exact sequence above can be given as
$$1 \to J_{K,S} \to J_{L,S}^{G} \to H^1(G,I_{L,S}) \to 1$$
which implies that
\begin{align*}
[J_{L,S}^{G}:J_{K,S}]&=[H^1(G,I_{L,S})]\\
&=\prod_{v\notin{S}}[H^1(G_w,U_w)]\\
&=\prod_{v\notin{S}}e(v)
\end{align*}
\end{proof}

Now note that
$$1 \to U_{L,S} \to L^{\times} \to J_{L,S} \to \Cl_{L,S} \to 1$$
is an exact sequence of $G$-modules. Denoting the image of $L^{\times}\to J_{L,S}$ by $P_{L,S}$, we obtain two short exact sequences of $G$-modules
\begin{equation}\label{eq1}
1 \to U_{L,S} \to L^{\times} \to P_{L,S} \to 1
\end{equation}
and
\begin{equation}\label{eq2}
1 \to P_{L,S} \to J_{L,S} \to \Cl_{L,S} \to 1
\end{equation}
The short exact sequence (\ref{eq2}) induces the long exact sequence
$$1 \to P_{L,S}^G \to J_{L,S}^G \to \Cl_{L,S}^G \to H^1(G,P_{L,S}) \to H^1(G,J_{L,S})$$
By Shapiro's lemma \cite[IX \S2, Lemma 3]{ant}, we have
$$ H^1(G,J_{L,S})\isom\bigoplus_{v\notin{S}}H^1(G_w,\Z)$$
where $G_w$ is the decomposition group at $w|v$. But $H^1(G_w,\Z)=1$, so that
$$H^1(G,J_{L,S})=1$$
Hence, the long exact sequence above reduces to
\begin{equation}\label{eq3}
1 \to J_{L,S}^G/P_{L,S}^G \to \Cl_{L,S}^G \to H^1(G,P_{L,S}) \to 1
\end{equation}
 The ideal classes in $\Cl_{L,S}^{G}$ are known as ambiguous $S$-ideal classes, and denoted by $\Am(L/K,S)$. In particular, for $G=\langle\sigma\rangle$, an ideal class $[\a]\in\Cl_{L,S}$ is ambiguous if $[\a]^{\sigma}=[\a]$, i.e. there exists $\alpha\in{L^{\times}}$ such that ${\a}^{\sigma-1}=(\alpha)$. If $\alpha\in{U_{L,S}}$, i.e. ${\a}^{\sigma-1}=(1)$, we say that the class $[\a]$ is \emph{strongly ambiguous}. The set of all strongly ambiguous $S$-ideal classes is denoted by $\Am_{\st}(L/K,S)$, and it is clear that 
$$\Am_{\st}(L/K,S)=J_{L,S}^G/P_{L,S}^G$$
The next lemma measures the difference between the ambiguous and the strongly ambiguous $S$-ideal classes. This may be considered as the $S$-version of \cite[Prop. 1]{lem}.

\begin{lemma}\label{lem2}
There is an exact sequence 
$$1 \to \Am_{\st}(L/K,S) \to \Am(L/K,S) \to \left(U_{K,S}\cap\Nm(L^{\times})\right)/\Nm(U_{L,S}) \to 1$$
\end{lemma}

\begin{proof}
By the exact sequence (\ref{eq3}), it suffices to show that
$$H^1(G,P_{L,S})\isom \left(U_{K,S}\cap\Nm(L^{\times})\right)/\Nm(U_{L,S}) $$
The exact sequence (\ref{eq1}) of $G$-modules induces, using Hilbert's Theorem 90, the cohomology sequence
$$1 \to H^1(G,P_{L,S}) \to H^2(G,U_{L,S}) \to H^2(G,L^{\times})$$
Since $G$ is cyclic, periodicity of the cohomology groups implies that the above sequence may also be given as
$$1 \to H^1(G,P_{L,S}) \to U_{K,S}/\Nm(U_{L,S}) \to K^{\times}/\Nm(L^{\times})$$
which immediately gives us the desired isomorphism.
\end{proof}

We have the following ambiguous $S$-class number formula, which may be compared with \cite[Thm. 1]{lem} as well as with \cite[XIII, \S4, Lem. 4.1]{cyclo}.

\begin{theorem}[Ambiguous $S$-class number formula]\label{thm1}
We have
\begin{align*}
[\Am(L/K,S)] &=\frac{h_{K,S}\;\;e_{\ssstyle{L/K,S}}}{[L:K][U_{K,S}:U_{K,S}\cap\Nm(L^{\times})]}\\
[\Am_{\st}(L/K,S)] &=\frac{h_{K,S}\;\;e_{\ssstyle{L/K,S}}}{[L:K][U_{K,S}:\Nm(U_{L,S})]}
\end{align*}
where
$$e_{\ssstyle{L/K,S}}=\prod_{v\in{S}}n_v\;\prod_{v\notin{S}}e(v)$$
with $n_v=[L_w:K_v]$ and $e(v)$ the ramification index of $v$.
\end{theorem}

\begin{proof}
 By Lemma \ref{lem2}, it suffices to prove the formula for $$[\Am_{\st}(L/K,S)]=[J_{L,S}^G:P_{L,S}^G]$$
To begin with, by Lemma \ref{lem3}, we have
$$[J_{L,S}^G:J_{K,S}]=\prod_{v\notin{S}}e(v)$$
On the other hand, the long exact sequence of cohomology groups induced by the short exact sequence (\ref{eq1}) gives the exact sequence
$$1 \to P_{K,S} \to P_{L,S}^G \to H^1(G,U_{L,S}) \to 1$$
where the $1$ on the right follows from Hilbert's theorem 90. This implies that
\begin{align*}
[P_{L,S}^G:P_{K,S}]&=h_1(U_{L,S})=\frac{h_2(U_{L,S})}{h_{2/1}(U_{L,S})}\\
&=\frac{[U_{K,S}:\Nm(U_{L,S})][L:K]}{\displaystyle\prod_{v\in{S}}n_v}
\end{align*}
where the last equality follows from \cite[IX \S5, Cor 2]{ant}. Thus we have
\begin{align*}
[\Am_{\st}(L/K,S)]&=[J_{L,S}^G:P_{L,S}^G]\\
&=\frac{[J_{L,S}^G:P_{K,S}]}{[P_{L,S}^G:P_{K,S}]}\\
&=\frac{[J_{L,S}^G:J_{K,S}][J_{K,S}:P_{K,S}]}{[P_{L,S}^G:P_{K,S}]}\\
&=h_{K,S}\;\frac{[J_{L,S}^G:J_{K,S}]}{[P_{L,S}^G:P_{K,S}]}\\
&=h_{K,S}\;\prod_{v\notin{S}}e(v)\;\frac{1}{[L:K][U_{K,S}:\Nm(U_{L,S})]}\;\prod_{v\in{S}}n_v\\
&=\frac{h_{K,S}\;\;e_{\ssstyle{L/K,S}}}{[L:K][U_{K,S}:\Nm(U_{L,S})]}
\end{align*}
\end{proof}

\section{Main Results}
We now prove the main results of this paper. 
We begin with an alternative description of the kernel and cokernel of the $S$-capitulation map.

\begin{theorem}\label{t6}
Suppose that $L/K$ is a finite cyclic extension of number fields with Galois group $G$, and let $S$ be a finite set of primes of $K$ containing the infinite primes. Let $\Cl_{K,S}$ and $\Cl_{L,S}$ denote the $S$-class groups of $K$ and $L$ respectively, and let $C_{K,S}$ and $C_{L,S}$ denote the corresponding $S$-id\`{e}le classes. Then there is an exact sequence
$$1 \to C_{L,S}^{G}/C_{K,S} \to \Cl_{K,S} \xrightarrow{j_{\ssstyle{L/K,S}}} \Cl_{L,S}^{G} \to H^1(G,C_{L,S}) \to 1$$
where $j_{\ssstyle{L/K,S}}$ is the $S$-capitulation map.
\end{theorem}

\begin{proof}
By Theorem \ref{t2}, there is an exact sequence
$$1 \to C_{L,S} \to C_L \to \Cl_{L,S} \to 1$$
of $G$-modules. The corresponding long exact sequence of cohomology yields
$$1 \to C_{L,S}^{G} \to C_L^G \to \Cl_{L,S}^{G} \to H^1(G,C_{L,S}) \to H^1(G,C_L)$$
By \cite[\S8, Prop 8.1]{gcft}, we have $C_L^G=C_K$ and, by \cite[\S9, Thm 9.1]{gcft}, we find that $H^1(G,C_L)$ is trivial. We thus get the exact sequence
$$1 \to C_{L,S}^{G} \to C_K \to \Cl_{L,S}^{G} \to H^1(G,C_{L,S}) \to 1$$
which can be modified to the exact sequence
$$1 \to C_{L,S}^{G}/C_{K,S} \to C_K/C_{K,S} \to \Cl_{L,S}^{G} \to H^1(G,C_{L,S}) \to 1$$
By Theorem \ref{t2} again, we have $C_K/C_{K,S}\isom\Cl_{K,S}$ so that the exact sequence becomes
$$1 \to C_{L,S}^{G}/C_{K,S} \to \Cl_{K,S} \xrightarrow{j_{\ssstyle{L/K,S}}} \Cl_{L,S}^{G} \to H^1(G,C_{L,S}) \to 1$$
\end{proof}

Thus we have
\begin{align*}
\ker{j_{\ssstyle{L/K,S}}} &\isom C_{L,S}^{G}/C_{K,S}\\
\coker{j_{\ssstyle{L/K,S}}} &\isom H^1(G,C_{L,S})
\end{align*}
In particular, the $S$-capitulation kernel measures the failure of Galois descent for the $G$-module $C_{L,S}$. If Galois descent holds for $C_{L,S}$, then the group $H^1(G,C_{L,S})$ measures the failure of Galois descent for $\Cl_{L,S}$.

The following corollary is immediate.

\begin{corollary}
The index $[C_{L,S}^{G}:C_{K,S}]$ divides $h_{K,S}$.
\end{corollary}

\begin{proposition}\label{p2}
Under the hypothesis of Theorem \ref{t6}, we have
$$[C_{K,S}:\Nm(C_{L,S})]=\frac{e_{\ssstyle{L/K,S}}}{[U_{K,S}:U_{K,S}\cap\Nm(L^{\times})]}$$
In particular, the index $[C_{K,S}:\Nm(C_{L,S})]$ divides $e_{\ssstyle{L/K,S}}$.
\end{proposition}

\begin{proof}
It follows from Theorem \ref{t5} that
$$[C_{L,S}^{G}:\Nm(C_{L,S})]=h_2(C_{L,S})=[L:K]\;h_1(C_{L,S})$$
On the other hand, the exact sequence in Theorem \ref{t6} implies that
$$h_1(C_{L,S})=[C_{L,S}^{G}:C_{K,S}]\;\frac{[\Cl_{L,S}^{G}]}{h_{K,S}}$$
Thus we get
$$[C_{L,S}^{G}:\Nm(C_{L,S})]=[C_{L,S}^{G}:C_{K,S}]\;[L:K]\;\frac{[\Cl_{L,S}^{G}]}{h_{K,S}}$$
Hence,
\begin{align*}
[C_{K,S}:\Nm(C_{L,S})] &=\frac{[C_{L,S}^{G}:\Nm(C_{L,S})]}{[C_{L,S}^{G}:C_{K,S}]}\\
&=[L:K]\;\frac{[\Cl_{L,S}^{G}]}{h_{K,S}}\\
&=\frac{e_{\ssstyle{L/K,S}}}{[U_{K,S}:U_{K,S}\cap\Nm(L^{\times})]}
\end{align*}
where the last equality follows from Theorem \ref{thm1}.
\end{proof}

We now study the cohomology of $U_{L,S}$, the $S$-unit group of $L$. The inclusions 
\begin{align*}
U_{L,S} &\subset L_w^{\times}\;\;\textrm{for}\;v\in{S}\\
\textrm{and}\;U_{L,S} &\subset U_w\;\;\textrm{for}\;v\notin{S}
\end{align*}
imply that, for $i=1,2$, we can define the localization maps
$$\lambda_{S}^i\;:\;H^i(G,U_{L,S})\to\bigoplus_{v\in{S}}H^i(G_w,L_w^{\times})\times\bigoplus_{v\notin{S}}H^i(G_w,U_w)$$
Define
\begin{align*}
\Sha_{S}^i(G,U_{L,S}) &:=\ker{\lambda_{S}^i}\\
\Cok_{S}^i(G,U_{L,S}) &:=\coker{\lambda_{S}^i}
\end{align*}
Theorem \ref{t3} implies that

\begin{lemma}\label{l2}
For $i=1,2$, there is an exact sequence
$$1 \to \Sha_{S}^i(G,U_{L,S}) \to H^i(G,U_{L,S}) \xrightarrow{\lambda_{S}^i} H^i(G,I_{L,S}) \to \Cok_{S}^i(G,U_{L,S}) \to 1$$
\end{lemma}

By Theorem \ref{t1}, there is an exact sequence of $G$-modules
$$1 \to U_{L,S} \to I_{L,S} \to C_{L,S} \to 1$$
We have $U_{L,S}^{G}=U_{K,S}$, and  Lemma \ref{l1} gives $I_{L,S}^{G}=I_{K,S}$. It follows that the induced long exact sequence of cohomology groups can be given as
\begin{align*}
1 &\to C_{L,S}^{G}/C_{K,S} \to H^1(G,U_{L,S}) \xrightarrow{\lambda_{S}^1} H^1(G,I_{L,S})\to H^1(G,C_{L,S})\\
& \to H^2(G,U_{L,S}) \xrightarrow{\lambda_{S}^2} H^2(G,I_{L,S}) \to H^2(G,C_{L,S})\\
& \to H^3(G,U_{L,S}) \to H^3(G,I_{L,S})
\end{align*}

Since the cohomology of cyclic groups is periodic, we can split up this long exact sequence into three exact sequences as 
\begin{equation}\label{e1}
1 \to C_{L,S}^{G}/C_{K,S} \to H^1(G,U_{L,S}) \xrightarrow{\lambda_{S}^1} H^1(G,I_{L,S}) \to \Cok_{S}^1(G,U_{L,S}) \to 1
\end{equation}

\begin{equation}\label{e2}
1 \to \Cok_{S}^1(G,U_{L,S}) \to H^1(G,C_{L,S}) \to \Sha_{S}^2(G,U_{L,S}) \to 1
\end{equation}

\begin{equation}\label{e3}
1 \to \Cok_{S}^2(G,U_{L,S}) \to H^2(G,C_{L,S}) \to \Sha_{S}^1(G,U_{L,S}) \to 1
\end{equation}

Lemma \ref{l2} and the exact sequence (\ref{e1}) immediately imply that

\begin{theorem}\label{t7}
There is an isomorphism
$$C_{L,S}^{G}/C_{K,S}\isom\Sha_{S}^1(G,U_{L,S})$$
\end{theorem}

Theorem \ref{t7} and Theorem \ref{t6} imply the following result, which also appears in \cite[\S1 Satz 2]{schmit} for the case $S=S_{\infty}$.

\begin{corollary}
There is an isomorphism
$$\ker{j_{\ssstyle{L/K,S}}}\isom\Sha_{S}^1(G,U_{L,S})$$
\end{corollary}

We now use the exact sequence (\ref{e1}) to calculate the order of $\Cok_{S}^1(G,U_{L,S})$. We have
\begin{align*}
[\Cok_{S}^1(G,U_{L,S})] &=[C_{L,S}^{G}:C_{K,S}]\;\frac{h_1(I_{L,S})}{h_1(U_{L,S})}\\
&=[C_{L,S}^{G}:C_{K,S}]\;\prod_{v\notin{S}}e(v)\;\frac{h_{2/1}(U_{L,S})}{h_2(U_{L,S})}\\
&=[C_{L,S}^{G}:C_{K,S}]\;\prod_{v\notin{S}}e(v)\;\frac{1}{[U_{K,S}:\Nm(U_{L,S})]}\;\frac{\displaystyle\prod_{v\in{S}}n_v}{[L:K]}
\end{align*}

We have thus proved that

\begin{theorem}\label{t8}
The group $\Cok_{S}^1(G,U_{L,S})$ is finite and its order is given by
$$[\Cok_{S}^1(G,U_{L,S})]=\frac{[C_{L,S}^{G}:C_{K,S}]}{[U_{K,S}:\Nm(U_{L,S})]}\;\frac{e_{\ssstyle{L/K,S}}}{[L:K]}$$
\end{theorem}

Suppose that $[L:K]=n$ and $d=\gcd(n,e_{\ssstyle{L/K,S}})$. Let $n'$ and $e'$ be the integers defined by
$$n'=n/d\;\;\textrm{and}\;\;e'={e_{\ssstyle{L/K,S}}}/d$$
By Theorem \ref{t8}, we have
\begin{equation}\label{e4}
n'\;[\Cok_{S}^1(G,U_{L,S})]\;[U_{K,S}:\Nm(U_{L,S})]=e'\;[C_{L,S}^{G}:C_{K,S}]
\end{equation}
Since $\gcd(n',e')=1$, it follows that $n'$ divides the index $[C_{L,S}^{G}:C_{K,S}]=[\ker{j_{\ssstyle{L/K,S}}}]$. We thus obtain the following generalization of Hilbert's Theorem 94, also proved in \cite{gon}.

\begin{theorem}\label{H94}
Assume the hypothesis of Theorem \ref{t6} and suppose that $L/K$ has degree $n$. Then there are at least $n'$ $S$-ideal classes in $K$ that become principal in $L$, where
$$n'=\frac{n}{\gcd(n,e_{\ssstyle{L/K,S}})}$$
In particular, $n'$ divides the $S$-class number $h_{K,S}$.
\end{theorem}

Using the exact sequence (\ref{e2}), Theorem \ref{t6}, and Theorem \ref{t8}, we get
\begin{align*}
[\Sha_{S}^2(G,U_{L,S})] &=\frac{[H^1(G,C_{L,S})]}{[\Cok_{S}^1(G,U_{L,S})]}\\
&=\frac{[\coker{j_{\ssstyle{L/K,S}}}]}{[\ker{j_{\ssstyle{L/K,S}}}]}\;\frac{[L:K][U_{K,S}:\Nm(U_{L,S})]}{e_{\ssstyle{L/K,S}}}\\
&=\frac{[\Am(L/K,S)]}{h_{K,S}}\;\frac{[L:K][U_{K,S}:\Nm(U_{L,S})]}{e_{\ssstyle{L/K,S}}}\\
&=\frac{[\Am(L/K,S)][L:K]}{h_{K,S}\;e_{\ssstyle{L/K,S}}}\;[U_{K,S}:\Nm(U_{L,S})]\\
&=\frac{[U_{K,S}:\Nm(U_{L,S})]}{[U_{K,S}:U_{K,S}\cap\Nm(L^{\times})]}
\end{align*}
where the last equality follows from Theorem \ref{thm1}.This shows that
\begin{theorem}\label{t9}
There is an isomorphism
$$\Sha_{S}^2(G,U_{L,S})\isom U_{K,S}\cap\Nm(L^{\times})/\Nm(U_{L,S})$$
\end{theorem}

Comparing Theorem \ref{t9} with Lemma \ref{lem2}, we obtain
\begin{corollary}
There is an exact sequence
$$1 \to \Am_{\st}(L/K,S) \to \Am(L/K,S) \to \Sha_{S}^2(G,U_{L,S}) \to 1$$
In particular, the group $\Sha_{S}^2(G,U_{L,S})$ measures the difference between the ambiguous $S$-ideal classes and the strongly ambiguous $S$-ideal classes.
\end{corollary}

Now note that the $S$-capitulation map $j_{\ssstyle{L/K,S}}:\Cl_{K,S}\to\Am(L/K,S)$ induces the map
$$j'_{\ssstyle{L/K,S}}:\Cl_{K,S}\to\Am_{\st}(L/K,S)$$
Applying \cite[Cor 2]{snake} to the composition
$$\Cl_{K,S}\to\Am_{\st}(L/K,S)\to\Am(L/K,S)$$
yields the isomorphism
$$\ker{j'_{\ssstyle{L/K,S}}}\isom\ker{j_{\ssstyle{L/K,S}}}$$
as well as the exact sequence
$$1\to\coker{j'_{\ssstyle{L/K,S}}}\to\coker{j_{\ssstyle{L/K,S}}}\to\Sha_{S}^2(G,U_{L,S})\to 1$$
Comparing this exact sequence with the exact sequence (\ref{e2}) and using Theorem \ref{t6}, we obtain

\begin{theorem}\label{T1}
There is an isomorphism
$$\coker{j'_{\ssstyle{L/K,S}}}\isom\Cok_{S}^1(G,U_{L,S})$$
\end{theorem}

We now use the exact sequence (\ref{e3}) and Theorem \ref{t7} to obtain
\begin{align*}
[\Cok_{S}^2(G,U_{L,S})]&=\frac{[H^2(G,C_{L,S})]}{[\Sha_{S}^1(G,U_{L,S})]}\\
&=\frac{[C_{L,S}^{G}:\Nm(C_{L,S})]}{[C_{L,S}^{G}:C_{K,S}]}
\end{align*}
This proves that

\begin{theorem}\label{t10}
There is an isomorphism
$$\Cok_{S}^2(G,U_{L,S})\isom C_{K,S}/\Nm(C_{L,S})$$
\end{theorem}

Using Proposition \ref{p2}, we calculate the order of $\Cok_{S}^2(G,U_{L,S})$.

\begin{corollary}\label{c}
We have
$$[\Cok_{S}^2(G,U_{L,S})]=\frac{e_{\ssstyle{L/K,S}}}{[U_{K,S}:U_{K,S}\cap\Nm(L^{\times})]}$$
\end{corollary}

The next result extends \cite[Thm 5.2]{gon}.

\begin{theorem}
Suppose that $L/K$ is a cyclic extension of degree $n$, and that $U_{K,S}\subset\Nm(L^{\times})$. Then
\begin{enumerate}
\item There are at least $e'$ ambiguous $S$-ideal classes in $L$ that do not arise from the $S$-ideal classes in $K$, where
$$e'=\frac{e_{\ssstyle{L/K,S}}}{\gcd(n,e_{\ssstyle{L/K,S}})}$$
\item We have
$$[\Cok_{S}^2(G,U_{L,S})]=[C_{K,S}:\Nm(C_{L,S})]=e_{\ssstyle{L/K,S}}$$
\end{enumerate}
\end{theorem}

\begin{proof}
Note that (2) follows immediately from Corollary \ref{c} since the hypothesis implies that $U_{K,S}\cap\Nm(L^{\times})=U_{K,S}$. In particular, (2) implies that both $\Cok_{S}^2(G,U_{L,S})$ and $C_{K,S}/\Nm(C_{L,S})$ are determined entirely by the primes that ramify in $L$. 

The same hypothesis also implies, by Theorem \ref{t9}, that 
$$\Sha_{S}^2(G,U_{L,S})\isom U_{K,S}/\Nm(U_{L,S})$$
The exact sequence (\ref{e2}) and Theorem \ref{t6} then imply that
$$[\coker{j_{\ssstyle{L/K,S}}}]=[H^1(G,C_{L,S})]=[\Cok_{S}^1(G,U_{L,S})][U_{K,S}:\Nm(U_{L,S})]$$
The equation (\ref{e4}) can now be given as
$$n'\;[\coker{j_{\ssstyle{L/K,S}}}]=e'\;[C_{L,S}^{G}:C_{K,S}]$$
As before, since $\gcd(n',e')=1$, it follows that $e'$ divides $[\coker{j_{\ssstyle{L/K,S}}}]$. This proves (1).
\end{proof}

Now recall \cite[\S7]{gon} that a nonempty set $S$ of primes of $K$ is \emph{large relative to $L/K$} if $S$ contains all the infinite primes of $K$ and all primes that ramify in $L/K$. In this case, we have the following extension of \cite[Thm 7.1]{gon}.

\begin{theorem}\label{L}
Suppose that $L/K$ is a finite, cyclic extension of number fields of degree $n$ with Galois group $G$, and let $S$ be a finite set of primes of $K$ that is large relative to $L/K$ (as defined above). Then
\begin{enumerate}
\item There is an exact sequence
$$1 \to H^1(G,U_{L,S}) \to \Cl_{K,S} \xrightarrow{j_{\ssstyle{L/K,S}}} \Cl_{L,S}^{G} \to \Sha_{S}^2(G,U_{L,S}) \to 1$$
\item The map 
$$j'_{\ssstyle{L/K,S}}:\Cl_{K,S}\to\Am_{\st}(L/K,S)$$ 
is surjective.
\item The index $[U_{K,S}:\Nm(U_{L,S})]$ is divisible by $e'$, where
$$e'=\frac{\displaystyle\prod_{v\in{S}}n_v}{\gcd(n,\displaystyle\prod_{v\in{S}}n_v)}$$
and $n_v=[L_w:K_v]$.
\end{enumerate}
\end{theorem}

\begin{proof}
Since $S$ is large relative to $L/K$, we have
$$\prod_{v\notin{S}}e(v)=1$$
so that
$$e_{\ssstyle{L/K,S}}=\prod_{v\in{S}}n_v$$
It then follows from Theorem \ref{t3} that $H^1(G,I_{L,S})$ is trivial. By Lemma \ref{l2}, we have
$$\Sha_{S}^1(G,U_{L,S})\isom H^1(G,U_{L,S})$$
By Theorems \ref{t6} and \ref{t7}, this implies that
$$\ker{j_{\ssstyle{L/K,S}}}\isom H^1(G,U_{L,S})$$
On the other hand, since $\Cok_{S}^1(G,U_{L,S})$ is also trivial, the exact sequence (\ref{e2}) implies that there is an isomorphism
$$H^1(G,C_{L,S})\isom\Sha_{S}^2(G,U_{L,S})$$
By Theorem \ref{t6}, this implies that
$$\coker{j_{\ssstyle{L/K,S}}}\isom\Sha_{S}^2(G,U_{L,S})$$
This proves (1).

Since $\Cok_{S}^1(G,U_{L,S})$ is trivial, Theorem \ref{T1} implies that $\coker{j'_{\ssstyle{L/K,S}}}$ is also trivial. This proves (2).

Finally, since $\Cok_{S}^1(G,U_{L,S})$ is trivial, equation (\ref{e4}) shows that we have
$$n'\;[U_{K,S}:\Nm(U_{L,S})]=e'\;[C_{L,S}^{G}:C_{K,S}]$$
where
$$e'=\frac{e_{\ssstyle{L/K,S}}}{\gcd(n,e_{\ssstyle{L/K,S}})}=\frac{\displaystyle\prod_{v\in{S}}n_v}{\gcd(n,\displaystyle\prod_{v\in{S}}n_v)}$$
It follows that the integer $e'$ divides the index $[U_{K,S}:\Nm(U_{L,S})]$, thus proving (3).
\end{proof}

\newpage
\newcommand{\etalchar}[1]{$^{#1}$}

\end{document}